\documentclass[10pt,bezier]{article}
\usepackage{amsmath,amssymb,amsfonts,graphicx,caption,subcaption}
\usepackage[colorlinks,linkcolor=red,citecolor=blue]{hyperref}

\newtheorem{prethm}{{\bf Theorem}}[section]
\newenvironment{thm}{\begin{prethm}{\hspace{-0.5
em}{\bf.}}}{\end{prethm}}
\newtheorem{prepro}{{\bf Theorem}}

\newtheorem{precor}[prethm]{{\bf Corollary}}
\newenvironment{cor}{\begin{precor}{\hspace{-0.5
em}{\bf.}}}{\end{precor}}
\newtheorem{preconj}[prethm]{{\bf Conjecture}}
\newenvironment{conj}{\begin{preconj}{\hspace{-0.5
em}{\bf.}}}{\end{preconj}}
\newtheorem{preremark}[prethm]{{\bf Remark}}

\newtheorem{prelem}[prethm]{{\bf Lemma}}
\newenvironment{lem}{\begin{prelem}{\hspace{-0.5
em}{\bf.}}}{\end{prelem}}
\newtheorem{preque}[prethm]{{\bf Question}}

\newtheorem{preobserv}[prethm]{{\bf Observation}}

\newtheorem{preproposition}[prethm]{{\bf Proposition}}

\newtheorem{preproof}{{\bf Proof.}}
\newtheorem{preprooff}{{\bf Proof}}

\newenvironment{proof}[1]{\begin{preproof}{\rm
#1}\hfill{$\Box$}}{\end{preproof}}
\newenvironment{prooff}[1]{\begin{preprooff}{\rm
#1}\hfill{$\Box$}}{\end{preprooff}}

\newtheorem{preproofs}{{\bf Second proof of }}

\newtheorem{preprooft}{{\bf Third proof of }}

\newtheorem{preproofF}{{\bf Proofs of}}

\title{\bf\Large 
Spanning closed walks 
with bounded maximum degrees of
 graphs on surfaces 
}
\author{{\normalsize{\sc Morteza Hasanvand${}$} }\vspace{3mm}
\\{\footnotesize{${}$\it Department of Mathematical
 Sciences, Sharif
University of Technology, Tehran, Iran}}
{\footnotesize{}}\\{\footnotesize{   $\mathsf{hasanvand@alum.sharif.edu  }$ }}}

\date{}
\begin{document}
\maketitle
\begin{abstract}{
 Gao and Richter (1994) showed that every $3$-connected  graph which embeds on the plane or the projective plane has a spanning closed walk meeting each vertex at most $2$ times.  Brunet, Ellingham, Gao, Metzlar, and Richter (1995)   extended   this result  to the torus and Klein bottle.  Sanders and Zhao (2001)  obtained a sharp result for higher  surfaces  by proving that every  $3$-connected graph embeddable on a surface with Euler characteristic $\chi \le  -46$  admits a spanning closed walk meeting each vertex at most $\lceil \frac{6-2\chi}{3}\rceil$ times. In this paper, we develop these  results   to the remaining surfaces with Euler characteristic  $\chi \le   0$. 
} {\small
}
\end{abstract}
%
%
%
%
%
%
%
%
%
%
%
%
%
%
\section{Introduction}
In this article, all graphs have  no  loop, but  multiple  edges are allowed and a simple graph is a graph without multiple edges.
 Let $G$ be a graph. 
The vertex set, the edge set, and the number of components of $G$ are denoted by $V(G)$, $E(G)$,  and $\omega(G)$, respectively. 
For a vertex set $S$ of  $G$, we denote by $e_G(S)$ the number of edges of $G$ with both ends in $S$.
Also, $S$ is called an {\it  independent} set, if there is no edge
of $G$ connecting vertices in $S$.
A {\it  minor} of $G$ refers to a  graph $R$ which can obtained from $G$ by contracting some vertex-disjoint connected subgraphs of $G$.
A {\it  $k$-walk} ({\it  $k$-trail}) in a graph refers to a spanning closed walk (trail)
meeting each vertex  at most $k$ times.
In  this paper, we assume that all walks use each edge at most two times.
Note that $1$-walks and $1$-trails of a graph are
equivalent to Hamilton cycles.
A graph $G$ is called {\it$m$-tree-connected}, if it has $m$ edge-disjoint spanning trees. 
It was known that the vertex set of any graph $G$ can  be expressed uniquely (up to order) as a disjoint union of vertex sets of some induced $m$-tree-connected subgraphs~\cite{MR1010580}.
These subgraphs are called the
 $m$-tree-connected components of $G$.
For a graph $G$,  we define $\Omega(G)=|P|-\frac{1}{2}e_G(P)$,  in which $P$ is the unique partition of $V(G)$
obtained from the $2$-tree-connected components of $G$ and $e_G(P)$ denotes the number of edges of $G$ joining different parts of $P$.
%
%
%
%
%
%
%

In 1956 Tutte~\cite{MR0081471} showed that every $4$-connected plane graph
admits a Hamilton cycle.
Later, Gao and Richter (1994)~\cite{MR1305052}  
proved that  every $3$-connected graph  which embeds on the plane or the projective plane admits a $2$-walk, which
 was  conjectured by Jackson and Wormald (1990)~\cite{MR1126990}.
In 1995 Brunet,
Ellingham, Gao, Metzlar, and Richter~\cite{MR1347338} extended  this result to the torus and  Klein bottle, and 
 also proposed the following conjecture.
This conjecture   
 is verified in~\cite{MR1411237, MR1644055} with linear bounds on  $k_\chi$ (but  not sharp).
\begin{conj}{\rm (\cite{MR1347338})}\label{conj:1}
For every integer $\chi$, there is a positive  integer  $k_\chi$ such that every $3$-connected graph which embeds on a surface with Euler characteristic $\chi$ admits a $k_\chi$-walk. 
\end{conj}
In 2001 Sander and Zhao~\cite{MR1805495} obtained a sharp  bound on $k_\chi$ for surfaces with small enough 
Euler characteristic   
and established the following  theorem.
In Section~\ref{sec:3},  we develop this  result   to the surfaces with Euler characteristic  $\chi \le   0$. 
Our proof  is based on a  recent result in~\cite{II} and inspired by some methods that introduced in~\cite{MR2959397,MR1805495}.
\begin{thm}{\rm (\cite{MR1805495})}\label{thm:Sanders-Zhao}
Every $3$-connected graph $G$ embeddable on a surface with Euler characteristic $\chi \le -46$
has a $\big\lceil \frac{6-2\chi}{3}\big\rceil$-walk.
\end{thm} 
%
In 1994
Gao and  Wormald~\cite{MR1289970} investigated trails in triangulations and derived the following theorem.
They also deduced that every $5$-connected triangulation  in the double  torus with
 representativity
at  least $6$  admits  a $4$-trail. 
In Section~\ref{sec:trails}, 
we conjecture that for each 
$\chi$, there is a positive  integer  $k_\chi$ such that every $4$-connected graph which embeds on a surface with Euler characteristic $\chi$ admits a $k_\chi$-trail, and  verify it for $5$-connected graphs.
As a consequence, we deduce that every $5$-connected graph  which embeds on  the double  torus (not necessarily triangulation) admits a $3$-trail.
\begin{thm}{\rm (\cite{MR1289970})}\label{thm:Frank,Gyarfas}
All  triangulations  in  the projective  plane,  the  torus  and  the  Klein  bottle 
have $4$-trails.
\end{thm}
%
%
%
%
\section{Minor minimal $3$-connected graphs having no $k$-walks}
The following theorem  is inspired by Theorem~5 in~\cite{MR2959397} and provides a common improvement for Lemmas~3.1 and~3.2 in~\cite{MR1805495}.
We will apply it in the next section.
\begin{thm}\label{thm:minor-minimal}
{Let $G$ be a $3$-connected graph and let $k$ be an integer with $k\ge 3$.
If  $G$ has no  $k$-walks, then 
$G$ contains a minor $3$-connected   bipartite graph  $R$ with the bipartition $(X,Y)$ with the following properties:
\begin{enumerate}{
\item   $R$ has no  $k$-walks.
\item  $X=\{v\in V(G):d_R(v)\ge k+1\}$.
\item$Y=\{v\in V(G):d_R(v)=3\}$.
}\end{enumerate}
}\end{thm}
Let $G$ be a $3$-connected graph. 
For an edge $e \in E (G )$, define $G /e$ to be the graph obtained from $G$ by
contracting $e$.
An edge $e$ is said to be {\it contractible} if $G /e$ is still $3$-connected.
For proving Theorem~\ref{thm:minor-minimal}, 
we require the following two lemmas, which the first one is well-known.
\begin{lem}{\rm(See Halin~\cite{MR0248042})}\label{lem:Halin:1}
{Let G be a $3$-connected graph except for $K_4$. 
Then every vertex of degree $3$ in $G$ is incident with a contractible edge.
}\end{lem}
\begin{lem}{\rm(Halin~\cite{MR0248042, MR0278979})}\label{lem:Halin:2}
{Let G be a minimally $3$-connected graph and define $V_3(G)=\{v\in V(G):d_G(v)=3\}$.
Then the
following statements hold:
\begin{enumerate}{
\item[{\upshape (i)}] $V_3(G) \neq \emptyset$.
\item[{\upshape (ii)}] Every edge connecting two vertices in $V(G) \setminus  V_3 (G)$ is contractible.
\item[{\upshape (iii)}]  The graph obtained by contracting any edge connecting two vertices in $V(G)\setminus
V_3(G)$ is also minimally $3$-connected.
\item[{\upshape (iv)}] Every cycle of G contains at least two vertices of $V_3(G)$.
}\end{enumerate}
}\end{lem}

\begin{prooff}{
\noindent
\hspace*{-4mm}
\textbf{ of Theorem \ref{thm:minor-minimal}.}
Let $G$ be a  counterexample with the minimum $|V(G)|$. 
We may assume that $G$ is   minimally $3$-connected
and also  $|V(G)|\ge 5$.
We here prove the following claim which was essentially
shown by Sanders and Zhao in \cite[Lemma 3.2]{MR1805495}.
%
\vspace{2mm}\\
{\bf Claim 1.} $\{v\in V(G):3\le d_G(v) \le k\}$ is independent.
Suppose otherwise  that there is an edge $xy\in E(G)$ such that $d_G(x) \le k$ and $d_G(y) \le k$.
If all edges incident with $x$ or $y$ are not contractible, then by Lemma~\ref{lem:Halin:1}, 
we must have  $d_G(x)>3$ and  $d_G(x)>3$, which contradicts Lemma~\ref{lem:Halin:2} (ii).
Thus we may assume that there is a contractible  edge incident with $y$, say $yz$ (possibly $z= x$).
By the minimality of $G$, the graph $G/yz$ has a $k$-walk $W$.
Define  $H$ to be  the  graph with the vertex set $V(G)$ having the same edges of $W$ by considering multiplicity of each edge. 
First, assume  $d_{H}(y)>0$. 
If both of $d_{H}(y)$ and $d_{H}(z)$ are odd, define $H'$ to be the graph obtained from $H$ by adding a  copy of $yz$;
otherwise,  define $H'$ to be the graph obtained from $H$ by adding two  copies of $yz$.
Next, assume  $d_{H}(y)=0$. 
In this case, define $H'$ to be the graph obtained from $H$ by adding two  copies of $xy$.
It is not difficult to check that  $H'$ is an Eulerian graph with $\Delta(H')\le 2k$. Hence $G$ admits a $k$-walk,
which is a contradiction.

We now prove the next claim.
\vspace{2mm}\\
{\bf Claim 2.}  $\{v\in V(G):d_G(v)\ge 4\}$ is independent.

Let $P=x_0x_1\ldots x_l$ be a maximal path in  the subgraph of $G$ induced by $\{v\in V(G):d_G(v)\ge 4\}$.
By Lemma~\ref{lem:Halin:2} (iv), this subgraph is a forest. 
Also,  $x_0$ and $x_l$ have degree one in it  and $x_ix_j\notin E(G)$, for any $0\le i< j\le l$ and $j\neq i+1$.
By applying Lemma~\ref{lem:Halin:2} (iii) repeatedly,  one can conclude that  $G/P$ is $3$-connected.
By the minimality of $G$, the graph $G/P$ has a $k$-walk $W$.
Define  $H$ to be  the graph with the vertex set $V(G)$ having the same edges of $W$ by considering multiplicity of each edge. 
Since $x_0$ and $x_l$ have degree one in $P$, there are two edges $x_0y_1, x_ly_2\in E(G)$ such that 
$y_1,y_2\in Y$.
Note that $\sum_{0\le i \le l}d_{H}(x_i)\le 2k$.
Let $x_j$ be a vertex of $V(P)$ with the maximum $d_{H}(x_j)$.
 If $d_{H}(x_j)\le 2k-3$,   define $H'$ to  be the graph obtained from $H$  by adding 
 some of the edges of a copy  of $P$ and adding another new copy of $P$ such that $H'$ forms an Eulerian graph. 
Otherwise, if $d_{H}(x_j) \ge 2k-2$, define  $H'$  to be the graph obtained from $H$
by adding a copy  of   the paths $y_1x_0Px_{i-1}$ (if $i\neq 1$) and $x_{i+1}Px_ly_2   $ (if $i\neq l$) 
and  adding some of the edges of $y_1Py_2$ such that $H'$ forms an Eulerian graph.
According to the construction, it is not difficult to check $\Delta(H')\le 2k+1$. Hence $G$ admits a $k$-walk, which is a contradiction again.

By the above-mentioned claims, $G$ is a  bipartite graph with the bipartition $(X,Y)$ in which 
$X=\{v\in V(G):d_R(v)\ge k+1\}$ and $Y=\{v\in V(G):d_R(v)=3\}$.
By taking $R=G$, we  derive that $G$ is not a counterexample and so the proof is completed.
}\end{prooff}
%
%
%
\section{$3$-connected graphs on surfaces}
\label{sec:3}
We shall below   develop  Theorem~\ref{thm:Sanders-Zhao} as mentioned in the Abstract.
For this purpose, we recall the following recent result from~\cite{II} that grantees the existence of walks with bounded maximum degrees
on specified independent vertex set.
\begin{lem}{\rm(\cite{II})}\label{lem:walk:tough}
{Let $G$ be a connected graph with the independent set $X\subseteq V(G)$ and 
let $k$ be a  positive integer.
Then  $G$ contains an spanning closed  walk meeting each  $v\in X$ at most $k$ times, 
if for every $S\subseteq X$, at least one of the  following  conditions holds:
\begin{enumerate}{
\item $\omega(G\setminus S)\le  (k-\frac{1}{2})|S|+1$.
\item $G$ contains an spanning closed  walk meeting each  $v\in S$ at most $k$ times.
}\end{enumerate}
}\end{lem}
Now, we are in a position to prove the main result of this paper.
\begin{thm}
Every $3$-connected graph $G$ embeddable on a surface with Euler characteristic $\chi \le 0$
admits a  $\big\lceil \frac{6-2\chi}{3}\big\rceil$-walk.
\end{thm} 
\begin{proof}
{We may assume that  $\chi < 0$, as the assertion  was  already proved in~\cite{MR1347338} for  the surfaces 
with  Euler characteristic zero  (namely the torus and Klein bottle).
By  Theorem~\ref{thm:minor-minimal}, we may also  assume that $G$ is a bipartite graph with the bipartition $(X,Y)$ such that 
  $X=\{v\in V(G):d_G(v)\ge k+1\}$ and 
$Y=\{v\in V(G):d_G(v)=3\}$, 
where $k=\lceil \frac{6-2\chi}{3} \rceil$.
Let  $S\subseteq X$ be a vertex cut  of $G$ so that $|S|\ge 3$.
Define 
$H$ 
to be the bipartite simple graph obtained from $G$
 by contracting any component of $G\setminus S$ such that one partite set is $S$.
Note that  $H$
 can be embedded on the surface as $G$ is embedded.
Since $G$ is $3$-connected, the minimum degree of $H$ is at least $3$.
Since 
$H$
 is triangle-free, by Euler's formula, it is easy to check that   $|E(H)| \le 2|V(G)|-2\chi$.
Thus
$$3\omega(G\setminus S) \le 
|E(H)| \le
 2|V(H)|-2\chi
 = 2 \big(|S|+\omega(G\setminus S)\big)-2\chi,$$
and so 
$$\omega(G\setminus S) \le 2|S|-2\chi \le (2+\frac{-2\chi -3/2}{3})|S| +3/2
\le \big(\lceil \frac{6-2\chi}{3}\rceil-1/2\big)|S|+3/2.$$
If  the equalities hold, the we must have  $|S|=3$ and $\omega(G\setminus S)=6-2\chi$, and also 
 every component of $G\setminus S$ has exactly $3$ neighbours in $S$.
In this case, we shall show that $G$ has a spanning closed  walk meeting each $v\in S$ at most $k$ times.
Since $G$ is $3$-connected, there is a cycle $C$ of $G$ containing all three vertices of $S$.
Since $S$ is independent, the cycle $C$ contains at least one vertex of exactly three components of $G\setminus S$.
For the remaining $3-2\chi$ components of $G\setminus S$, 
join every of them to $C$ such that each vertex in $S$ has  degree at most  $\lceil (3-2\chi)/3\rceil +2$.
Finally, add the edges of any component of $G\setminus S$ to  this new graph.
Obviously, the resulting spanning subgraph of $G$ has a spanning closed  walk meeting each $v\in S$ at most $k$ times and so does $G$.
Therefore, by Lemma~\ref{lem:walk:tough},  the graph $G$  have a spanning closed  walk meeting each $v\in X$ at most $k$ times.
Since $k\ge 3$ and every vertex of $Y$ in $G$ has degree $3$, this walk forms  a $k$-walk for $G$. Hence the theorem holds.
}\end{proof}
%

%
%
\section{Graphs with higher connectivity}
\label{sec:trails}
As we already mentioned, Tutte~\cite{MR0081471} proved that every $4$-connected planar graph
admits a Hamilton cycle. Thomas and Yu~\cite{MR1290634} extended this result to the  projective plane and Gr\"unbaum~\cite{MR0266050} and independently Nash-Williams~\cite{MR0387097} conjectured that Tutte's result could be  developed to  the torus.
Motivated by Conjecture~\ref{conj:1} and  these results,
we now propose the following conjecture.
By considering the complete bipartite graph $K_{4, -\chi+4}$ (which
can be embedded on a surface with Euler characteristic $\chi$),
if  this conjecture would be true,
 then we must have  $k_\chi \ge \lceil \frac{4-\chi}{4} \rceil$.
\begin{conj}\label{conj;2}
For every integer $\chi$, there is a positive  integer  $k_\chi$ such that every $4$-connected graph 
which embeds on a surface with Euler characteristic $\chi$ admits a $k_\chi$-trail. 
\end{conj}
We shall below prove Conjecture~\ref{conj;2} for $5$-connected graphs.
For this purpose, we require the following  recent result which gives a sufficient condition for the existence of $k$-trails.
\begin{lem}{\rm(\cite{II})}\label{lem:trail:tough}
{Let $G$ be a graph and let $k$ be a positive integer.
 If for all $S\subseteq V(G)$, 
$$\Omega(G\setminus S)\le  (k-\frac{1}{2})|S|+1-\frac{1}{2}e_G(S),$$
then $G$ admits a $k$-trail.
}\end{lem}
Now, we are ready to prove the main result of this section.
\begin{thm}
{Every  $5$-connected graph $G$ which embeds on a surface with  Euler characteristic $\chi  \le 0$ 
admits a  $\lceil \frac{6-3\chi}{4} \rceil$-trail.
}\end{thm}
%
%
\begin{proof}
{Let $S\subseteq V(G)$ with $\Omega(G\setminus S) > 1$. 
For each vertex $v$,  we have $\Omega(G- v)= 1$, 
because the graph $G-v$ is  $2$-tree-connected~\cite{MR0133253,MR0140438}.
Thus $|S|\ge 2$.
Define 
$H$ 
to be  the simple graph  obtained  form $G$ by contracting the $2$-tree-connected components of $G\setminus S$.
Set $S'=V(H)\setminus S$.
Note that 
$|V(H)|= |S|+|S'|\ge 3$ 
and
 $H$
 can be embedded on the surface as $G$ is embedded.
By Euler's formula, it is easy to check that  
 $|E(H)| \le 3|V(H)|-3\chi$.
Note also that there is no a  pair of edges of $G$ joining two different $2$-tree-connected components of $G\setminus S$.
Since $G$ is $5$-connected,  the minimum degree of $H$ is at least $5$. Thus 
$$5|S'|\le |E(H)|-e_H(S)+e_H(S'),$$
and so
$$5|S'|-e_H(S') \le |E(H)| -e_H(S)\le  3|V(H)|-3\chi-e_H(S)=   3(|S'|+|S|)-3\chi-e_G(S).$$
Therefore,
$$\Omega(G\setminus S)
= |S'|-\frac{1}{2}e_H(S')\le \frac{3}{2}|S|+\frac{-3\chi}{2} -\frac{1}{2}e_G(S).$$
and so
$$\Omega(G\setminus S) \le (\frac{3}{2}+\frac{-3\chi/2-1}{2})|S|+1-\frac{1}{2}e_G(S)
 \le (\lceil \frac{6-3\chi}{4}\rceil -\frac{1}{2})|S|+1-\frac{1}{2}e_G(S) .$$
Hence the theorem follows from  Lemma~\ref{lem:trail:tough} with $k=\lceil \frac{6-3\chi}{4}\rceil $.
}\end{proof}
%
%
The following result improves Corollary 2 in~\cite{MR1289970}, which 
says that every $5$-connected triangulation  in the double  torus with  representativity
at  least $6$  admits  a  $4$-trail. Note that  the Euler characteristic of the double torus is $-2$.
\begin{cor}
{Every  $5$-connected graph $G$ which embeds on the double torus admits a  $3$-trail.
}\end{cor}
%
%
%
%
%
%
%
%
%
%
%
%
%
%
%

%

\end{document}